\documentclass[twoside,12pt]{amsart}
\usepackage{amsmath,latexsym,amssymb,mathptm, times,verbatim, enumerate,mathcomp,lscape}
    \usepackage{longtable}\usepackage{tikz,stmaryrd}
\usepackage{pdfpages}

\def\m{\mathfrak m}
\def\a{\mathfrak a}
\def\b{\mathfrak b}

\newtheorem{theorem}{Theorem}[section]
\newtheorem{lemma}[theorem]{Lemma}
\newtheorem{corollary}[theorem]{Corollary}
\newtheorem{proposition}[theorem]{Proposition}

\newtheorem{proposition-no-advance}[equation]{Proposition}
\newtheorem{observation}[theorem]{Observation}
\newtheorem{construction}[theorem]{Construction}
\newtheorem{set-up}[theorem]{Set-up}

\newtheorem{quick consequences}[theorem]{Quick Consequences}
\newtheorem{question}[theorem]{Question}

\newtheorem{hypothesis-no-advance}[equation]{Hypothesis}
\newtheorem{claim-no-advance}[equation]{Claim}
\newtheorem{observation-no-advance}[equation]{Observation}

\theoremstyle{definition}
\newtheorem{facts and definitions}[theorem]{Facts and Definitions}
\newtheorem{definition}[theorem]{Definition}
\newtheorem{definition-no-advance}[equation]{Definition}
\newtheorem{setup}[theorem]{Setup}

\newtheorem{remark-no-advance}[equation]{Remark}
\newtheorem{remarks-no-advance}[equation]{Remarks}

\newtheorem{convention}[theorem]{Convention}

\newtheorem{careful calculation}[theorem]{Careful Calculation}

\newtheorem{present summary}[theorem]{Present Summary}

\newtheorem{example}[theorem]{Example}

\newtheorem{further reductions}[theorem]{Further Reductions}

\newtheorem{marching orders}[theorem]{Marching Orders}

\newtheorem{note}[theorem]{Note}

\begin{document}

\title[Totally Acyclic Complexes  Over Connected Sums with $\m ^3=0$]{Totally Reflexive Modules Over Connected Sums with $\m ^3 =0$}

\date{\today}

\author[A.~Vraciu]{Adela~Vraciu}
\address{Adela~Vraciu\\ Department of Mathematics\\ University of South Carolina\\
Columbia\\ SC 29208\\ U.S.A.} \email{vraciu@math.sc.edu}


\begin{abstract}
We give a criterion for rings with $\m^3=0$ which are obtained as connected sums of two other rings to have non-trivial totally acyclic modules.
\end{abstract}
\maketitle

\section{Introduction}

\begin{convention}
The rings in this paper are Noetherian standard graded algebras over a field $k$.
We will use $[R]_i$ to denote the $i$th graded component of $R$, and $\m_R$ will denote the unique maximal homogeneous ideal of $R$.

A complex  $\cdots \rightarrow  R^{b_i}\buildrel{d_i}\over\rightarrow R^{b_{i-1}}\rightarrow \cdots $ of free modules is called {\em minimal} if $\mathrm{im}(d_i)\subseteq \m_R R^{b_{i-1}}$ for all $i$.

$(\ )^*$ denotes the functor $\mathrm{Hom}_R(\ , R)$, and is called {\em the dual}.

\end{convention}

Totally reflexive modules were introduced in \cite{AB}:
\begin{definition}
A finitely generated module $M$ is {\em totally reflexive} if it is isomorphic to a syzygy in a doubly infinite exact complex of free $R$-modules
$$\mathcal{F}_{\cdot}: \cdots\buildrel{d_{i+1}}\over \rightarrow  R^{b_i}\buildrel{d_i}\over\rightarrow R^{b_{i-1}}\buildrel{d_{i-1}}\over\rightarrow\cdots,
$$
such that the dual $\mathcal{F}^*_{\cdot}$ is also exact. Such a complex is called {\em totally acyclic}.

Equivalently, $M$ is totally reflexive if $\mathrm{Ext}_R^i(M, R)=\mathrm{Ext}^i_R(M^*, R)=0$ for all $i \ge 1$, and $M \cong M^{**}$.
\end{definition}

A ring $R$ is Gorenstein if and only if the totally reflexive $R$-modules are precisely the maximal Cohen-Macaulay modules. Totally reflexive modules play an important role in the theory of Gorenstein dimension, which is a generalization of projective dimension.

Exact zero divisors provide a particularly simple example of totally reflexive modules:
\begin{definition}
A pair of elements $a, b \in R$ is a pair of  {\em exact zero divisors} if $\mathrm{ann}_R(a)=(b)$ and $\mathrm{ann}_R(b)=(a)$. Then $R/(a)$ and $R/(b)$ are totally reflexive modules, and
$$
\cdots \rightarrow R \buildrel{a}\over\rightarrow R \buildrel{b}\over\rightarrow R \buildrel{a} \over\rightarrow \cdots
$$
is a totally acyclic complex.
\end{definition}

The following result motivates the investigation in this paper.
\begin{theorem}[\cite{CPST}, Theorem 4.3]
Assume that $R$ is not Gorenstein. Then there are either infinitely many isomorphism classes of indecomposable totally reflexive  modules, or the only totally reflexive  modules are free.
\end{theorem}
 Note that existence of non-free totally reflexive modules is equivalent to existence of minimal totally acyclic complexes.
\begin{definition}
A ring $R$ is called {\em G-regular} if the only totally reflexive  modules are the free modules. 
\end{definition}

There is no known criterion for deciding if a given non-Gorenstein ring is G-regular or not. In the case when $\m_R^3=0$, the following conditions are proved to be necessary  for the existence of minimal totally acyclic complexes:
\begin{theorem}[\cite{Yo}, Theorem 3.1]\label{Yoshino}
Let $(R, \m_R)$ be such that $R$ is not Gorenstein and $\m_R^3=0$. Assume that $R$ is not G-regular. Then:

a. $R$ is isomorphic to a graded $k$-algebra $k \oplus[ R]_1 \oplus[ R]_2$ and Koszul; in particular, the defining ideal of $R$ is generated by polynomials of degree 2.

b. $\mathrm{dim}_k([R]_2)=\mathrm{dim}_k([R]_1)-1$

c. If $\cdots \rightarrow  R^{b_i}\buildrel{d_i}\over\rightarrow R^{b_{i-1}}\rightarrow \cdots $ is a minimal totally acyclic complex, then $b_i=b_{i-1}$ for all $i$, and the maps $d_i$ are represented by matrices with entries in $[R]_1$.
\end{theorem}

Even for rings with $\m_R ^3=0$, there are no known necessary and sufficient conditions for G-regularity.
In \cite{AV}, it was shown that rings obtained from Stanley-Reisner rings of graphs after modding out by a linear system of parameters satisfy $\m_R^3=0$, and some conditions for G-regularity of such rings were studied. Example (4.1) in \cite{AV} prompted us to consider the class of rings studied in this paper.

Fiber product rings have come to the attention of homological commutative algebraists in recent years. It was shown in~\cite{NSW} that if $\mathrm{Tor}^R_i(M, N)=0$ for all $i \gg 0$, where $M$ and $N$ are finitely generated modules over a ring $R$ which is a local Artinian fiber product ring over a field, then at least one of $M$ or $N$ is free. Since the condition $\mathrm{Ext}^i_R(M, R)=0$ in the definition of a totally reflexive module is equivalent by Matlis duality to $\mathrm{Tor}_i(M, \omega _R)=0$, where $\omega _R$ is the canonical module of $R$ (see Observation 2.10.2 in \cite{Kustin-V}), it follows that every such ring is either Gorenstein or G-regular.

Fiber product rings can be characterized by the condition that the maximal ideal is decomposable, i.e. $\m_R=\a \oplus \b$ for some ideals $\a$, $\b$.
In this paper, we look at rings with the property that $\m_R=\a + \b$ for some ideals $\a, \b$ with $\a \cdot \b =(0)$ and $\a \cap \b = (\delta )$, with $\delta \in \m_R^2$  (if $\delta \in \m_R \, \backslash \m_R^2$, we could write $\m_R$ as a direct sum of some smaller ideals $\a' $, $\b'$). We show that such rings can be obtained as quotients of fiber products by one element. These rings are connected sums in the sense of~\cite{AAM}.
We study the existence of totally reflexive modules for such rings under the additional assumption that $\m_R^3=0$  in terms of the existence of totally reflexive modules for the two rings involved in the fiber product.

Numerous examples of such rings can be obtained from graphs. Let $\Gamma $ be a connected bipartite graph with vertex set $\{{\bf x}_1, \ldots, {\bf x}_n, {\bf y}_1, \ldots, {\bf y}_m\}$ such that every edge connects an ${\bf x}_i$ to a ${\bf y}_j$. Assume that the induced graph on $\{{\bf x}_1, \ldots, {\bf x}_{n-1}, {\bf y}_1, \ldots, {\bf y}_{m-1}\}$ is disconnected and it has two connected components, $A$ and $B$. Also assume that ${\bf x}_n$ and ${\bf y}_m$ are not connected by an edge.

Let $R_{\Gamma}$ denote the Stanley-Reisner ring of $\Gamma$ over a fixed field $k$, and $R=R_{\Gamma}/(l_1, l_2)$, where $l_1=\sum_{i=1}^n X_i$  and $l_2=\sum_{j=1}^m Y_j$. We can view $R$ as a quotient of $k[X_1, \ldots, X_{n-1}, Y_1, \ldots, Y_{m-1}]$. It was shown in \cite{AV} that $(R, \m)$ has $\m_R^3=0$. Let $\mathfrak{a}$ denote the ideal generated by the images of variables corresponding to vertices in $A$ and $\mathfrak{b}$ the ideal generated by the images of variables corresponding to vertices in $B$. We have $\m = \a + \b$ and $\a \cdot \b =(0)$. Let $f=\sum_{i=1}^{n-1}x_i=-x_n$ and $g=\sum_{j=1}^{m-1}y_j=-y_m$, where $x_i$ and $y_j$ denote the images of $X_i$ and respectively $Y_j$ in $R$. Since ${\bf x}_n$ and ${\bf y}_m$ are not connected by an edge,  we have $fg=0$. We  write $f=f_A+f_B$, $g=g_A+g_B$, where $f_A$ is the sum of the $x_j$'s that are in $A$, etc.

We have $0=fg=f_Ag_A+f_Bg_B$  therefore, $\delta :=f_Ag_A=-f_Bg_B \in \a \cap \b$. There are no other elements in $\a\cap \b$. Since $\m_R ^3=(0)$, a non-zero element in the intersection would have to be $\displaystyle \sum_{x_i, y_j \in A}  x_iy_j= \sum _{x_i', y_j' \in B} x_i'y_j'$. Inspecting the defining equations of the Stanly-Reisner ring, we see that no such relation exists other than $fg=0$.

Proposition 3.9 in \cite{AV} shows that rings obtained from the constructin described above do not have exact zero divisors. On the other hand, Example 4.1 in \cite{AV} is an example of such a ring that has non-free totally reflexive modules. The rings studied in this paper can be viewed as generalizations of this example.

\section{Construction and set up}

\begin{observation}\label{setup}
 The following are equivalent:

{\rm 1.} The  maximal homogeneous ideal $\m_R$ can be decomposed as $\m_R=\mathfrak{a}+\mathfrak{b}$ with $\mathfrak{a}\mathfrak{b}=(0)$ and $\mathfrak{a} \cap \mathfrak{b}=(\delta_1, \ldots, \delta _s)$.

{\rm 2.}  $R$ is isomorphic to a ring of the form
\begin{equation}\label{def}
\frac{P}{I_1P+I_2P+(f_1-g_1, \ldots, f_s-g_s)+(x_iy_j \, | \, 1\le i \le n, 1 \le j \le m)}
\end{equation}
where $P=k[x_1, \ldots, x_n, y_1, \ldots, y_m]$, $I_1, I_2$ are ideals in $P_1:=k[x_1, \ldots, x_n]$, respectively $P_2:=k[y_1, \ldots, y_m]$, $f_1, \ldots, f_s \in P_1, g_1, \ldots, g_s \in P_2.$
\end{observation}
\begin{proof}
Assume {\rm 1.} 
Write
$R=P/J$ with $P=k[x_1, \ldots, x_n, y_1, \ldots, y_m]$, $\mathfrak{a} =(x_1, \ldots, x_n)$, $\mathfrak{b}=(y_1, \ldots, y_m)$. The assumption that $\mathfrak{a} \mathfrak{b}=(0)$ shows that $J_0:=(x_iy_j\, | \, 1\le i \le n, 1\le j \le m)\subseteq J$. We have canonical homomorphisms $P_1 \rightarrow R$ and $P_2 \rightarrow R$;  let $I_1$, respectively $I_2$ denote the kernels of these homomorphisms. Then $I_1P+I_2P \subseteq J$.  Let $R_0:=P_1/I_1$ and $S_0:=P_2/I_2$; these are isomorphic to subrings of $R$. We identify elements in $R_0, S_0$ with their images in $R$.
Modulo $J_0$, every element of $P$ can be written as $f-g$ with $f \in P_1$ and $g \in P_2$. 
Thus we  write $J=I_1P+I_2P+J_0 + (f_1-g_1, \ldots, f_t-g_t)$ for some $f_1, \ldots, f_t\in P_1, g_1, \ldots, g_t \in P_2$. We may assume that all $f_j, g_j$ are nonzero (if $f_j=0$, then $g_j \in J\Leftrightarrow g_j \in I_2P$). Note that $f_j$ and $g_j$ have the same image in $R$, which is therefore in $\mathfrak{a} \cap \mathfrak{b}$. A minimal generating set $f_1-g_1, \ldots, f_t-g_t$ for $J/I_1P+I_2P+J_0$ corresponds to a minimal generating set of $\mathfrak{a} \cap \mathfrak{b}$, thus $t=s$.

The proof of the converse follows along similar lines.
\end{proof}

\begin{note}
The ring $R$ described in~(\ref{def}) is a quotient of a fiber product:
$$
R=\frac{R_0 \times_k S_0}{(f_1-g_1, \ldots, f_s-g_s)},
$$
where $R_0=P_1/I_1, S_0=P_2/I_2$, and 
$$\displaystyle R_0 \times _k S_0=\frac{P}{I_1P+I_2P+(x_iy_j \, | \, 1 \le i \le n, 1\le j \le m)}.
$$
is the fiber product of $R_0$ and $S_0$ over $k$.
By abusing notation, we use $f_1, \ldots, f_s$ to denote the images of $f_1, \ldots, f_s \in P_1$ in $R_0$. Similarly for $g_1, \ldots, g_s$.
\end{note}
\begin{note}
If $f_1, \ldots, f_s \in \mathrm{Soc}(R_0)$ and $g_1, \ldots, g_s \in \mathrm{Soc}(S_0)$, then $R$ is a connected sum in the sense of~\cite{AAM}.
\end{note}

Connected sums of Gorenstein rings have received a lot of attention lately (see ~\cite{AAM}, \cite{ACLY}, \cite{CLW}). However, the connected sums we study in this paper are  non-Gorenstein.

We will focus on the case $s=1$. The following  notation will be in effect for the rest of the paper.

\begin{setup}\label{st-up}
Let $R$ be as in ~(\ref{def}), with $s=1$. Assume moreover that $\m_R^3=0$, and $f:=f_1, g:=g_1$  are nonzero elements of $R_0$, respectively $S_0$ of  degree two.

Denote $$\displaystyle R_0=\frac{P_1}{I_1}, \ \   R_1=\frac{P_1}{I_1+(f)}, \ \  S_0=\frac{P_2}{I_2}, \ \  S_1=\frac{P_2}{I_2+(g)},$$ $$\mathfrak{a}=(x_1, \ldots, x_n)R, \ \ \ \ \mathfrak{b}=(y_1, \ldots, y_m)R.$$

We have injective homormorphisms $\phi_1: R_0\rightarrow R$ and $\phi_2: S_0\rightarrow R$ induced by the inclusions $P_1 \subseteq P$ and $P_2 \subseteq P$. We will identify $R_0$ with $\mathrm{im}(\phi_1)$, which is the subring of $R$ generated by $\a$, and $S_0$ with $\mathrm{im}(\phi_2)$, which is the subring of $R$ generated by $\b$.
Note that $m_R^3=0 \Leftrightarrow \a ^3=\b^3=0 \Leftrightarrow \m_{R_0}^3=\m_{S_0}^3=0$. 

Assume $d:R^b \rightarrow  R^c$ is a degree one homomorphism of graded $R$-modules. There is a matrix  representation of $d$ of the form $A'+B'$, where $A'$ is a $c\times b$ matrix with entries in $\mathfrak{a}$ and $B'$ is a $c\times b$ matrix with entries in $\mathfrak{b}$.

 We can view $A'$ as a map $\mathrm{im}(\phi_1) ^b \rightarrow \mathrm{im}(\phi_1)^c$, and $B'$ as a map $: \mathrm{im}(\phi_2)^b \rightarrow \mathrm{im}(\phi_2)^c$. When $R_0$ is identified with $\mathrm{im}(\phi_1)$ and $S_0$ is identified with $\mathrm{im}(\phi_2)$, $A'$ and $B'$ correspond to maps $\tilde{A}:R_0^b\rightarrow R_0^c$ and $\tilde{B}:S_0^b \rightarrow S_0^c$ respectively.

 The assumption $\m_R^3=0$ guarantees that $\tilde{A}$ and $\tilde{B}$ map every element of degree two to zero, and therefore there are  induced maps $A:R_1^b \rightarrow R_1^c$, and $B:S_1^b \rightarrow S_2^c$. 

The process can be reversed as follows:  given maps $A:R_1^b \rightarrow R_1^c$ and $B:S_1^b \rightarrow S_1^c$ which are graded homomorphisms of degree one, there are  unique liftings $\tilde{A}: R_0^b \rightarrow R_0^c$ and $\tilde{B}:S_0^b \rightarrow S_0^c$ which map $f$ and $g$ to zero, and these can be identified with $c\times b$ matrices $A'$ and $B'$ with entries in $\mathfrak{a}$ and respectively $\mathfrak{b}$, giving rise to a homomorphism $d:R^b \rightarrow R^c$ represented by the matrix $A'+B'$.

Similarly, a vector in $R^b$ can be written  (uniquely, if all entries are linear) as $x'+y'$ where $x'$ has all components in $\a$ and $y'$ has all components in $\b$. These are identified with vectors $\tilde{x} \in R_0^b$ and $\tilde{y} \in S_0^b$. The images of $\tilde{x}$ in $R_1^b$ and of $\tilde{y}$ in $S_1^b$ will be denoted $x$ and $y$ respectively.

\end{setup}
\begin{observation}
$R$ is Gorenstein if and only if $R_0$ and $S_0$ are Gorenstein. 
\end{observation}
\begin{proof} 
Note that our assumptions imply $\delta \in \mathrm{Soc}(R)$. 
Assume that $R$ is Gorenstein. If $x' \in \mathrm{Soc}(R_0)$, then the image of $x'$ in $R$ must be in $(\delta)$, and therefore $x'\in (f)$, which shows that $R_0$ is also Gorenstein. The argument for $S_0$ is similar.

Now assume that $R_0$ and $S_0$ are Gorenstein.
Every element of $\m _R$ can be represented as $x'+y'$ with $x'\in \a$ and $y' \in b$. According to the convention in~(\ref{setup}), $x' \in \a$ corresponds to an element $\tilde{x} \in R_0$ and $y'\in \b$ corresponds to an element $\tilde{y}\in S_0$,  We have
$x'+y' \in \mathrm{Soc}(R) \Leftrightarrow \a x'=\b y'=0\Leftrightarrow \tilde{x} \in \mathrm{Soc}(R_0), \tilde{y} \in \mathrm{Soc}(S_0)$. Indeed, $x'+y' \in \mathrm{Soc}(R)$ implies that $ax'=-by' \in (\delta)$ for every choice of $a \in \a$ and $b \in \b$, and this can only happen if  $\a x'=\b y'=0$. Therefore, $\tilde{x}\in \mathrm{Soc}(R_0)= (f), \tilde{y}\in \mathrm{Soc}(S_0)= (g)$, which implies $x'+y' \in (\delta)$.
\end{proof}

From this point on, we will assume that $R$ is not Gorenstein. 

We will think of $R_1, S_1$, and choices of generators for their defining ideals as the data from which $R$ is constructed. 

\begin{construction}\label{constr2}
Given rings $R_1=P_1/(a_1, \ldots, a_t), S_1=P_2/(b_1, \ldots, b_u)$ with $m_{R_1}^3=\m_{S_1}^3=0$, we  let $I_1=\m_{P_1} a_1+(a_2, \ldots, a_t), I_2=\m_{P_2} b_1+(b_2, \ldots, b_u)$ and define $R$ to be the ring given by~(\ref{def}), with $s=1$, $f:=a_1, g:=b_1$.
\end{construction} Note that $\m_{P_1}a_1=0$ and $m_{P_2}b_1=0$ are redundant in the defining equations of $R$ (since they follow from $\m_{P_1}b_1=0$ and $\m_{P_2}a_1=0$). The same $R$ would be obtained by using $I_1=(a_2, \ldots, a_t)$, $I_2=(b_2, \ldots, b_u)$ in~(\ref{def}). However, the choice $I_1=\m_{P_1} a_1+(a_2, \ldots, a_t), I_2=\m_{P_2} b_1+(b_2, \ldots, b_u)$ guarantees that $R_0:=P_1/I_1$ and $S_0:=P_2/I_2$ satisfy $\m _{R_0}^3=\m_{S_0}^3=0$.

\begin{example}\label{exnew}
Let $$R_1=\frac{k[x_1, y_1, z_1]}{(x_1^2, y_1^2, z_1^2, x_1y_1)}, \ \ \ \ \ \ S_1=\frac{k[x_2, y_2, z_2]}{(x_2^2, y_2^2, z_2^2, x_2y_2)}.$$
We use the construction given in~(\ref{constr2}), using $f=z_1^2, g=z_2^2$.

The resulting ring is 
$$
R=\frac{k[x_1, y_1, z_1, x_2, y_2, z_2]}{(x_1, x_2, y_1, y_2)^2+z_1(x_2, y_2, z_2)+z_2(x_1, y_1, z_1)+(z_1^2-z_2^2)}
$$
\end{example}
\bigskip

\section{Main Results}
We study conditions on $R_1$ and $S_1$ that are necessary and sufficient for $R$ to admit minimal totally acyclic complexes.

More precisely, consider a sequence of maps
\begin{equation}\label{cx1}
\cdots \rightarrow R^{b_{i+1}}\buildrel{d_{i+1}}\over\rightarrow R^{b_i}\buildrel{d_i}\over\rightarrow R^{b_{i-1}}\buildrel{d_{i-1}}\over\rightarrow  \cdots
\end{equation}
and the induced sequences (recalling the notation from ~(\ref{st-up})):
\begin{equation}\label{cx2}
\cdots \rightarrow R_1^{b_{i+1}}\buildrel{A_{i+1}}\over\rightarrow R_1^{b_i}\buildrel{A_i}\over\rightarrow R_1^{b_{i-1}}\buildrel{A_{i-1}}\over\rightarrow \ldots   \ \ \ \ \ \mathrm{and} \ \ \ \ \ \ \ \cdots \rightarrow S_1^{b_{i+1}}\buildrel{B_{i+1}}\over\rightarrow S_1^{b_i}\buildrel{B_i}\over\rightarrow S_1^{b_{i-1}}\buildrel{B_{i-1}}\over\rightarrow \ldots 
\end{equation}
Conversely, given the maps in~(\ref{cx2}), we construct the maps in~(\ref{cx1}) by letting $d_i=A'_i + B'_i$ (where $A_i' , B_i': R^{b_i} \rightarrow R^{b_{i-1}}$ are obtained by lifting $A_i, B_i$ to $R_0^{b_i}$ and $S_0^{b_i}$ respectively, and then identifying $R_0, S_0$ with subrings of $R$).

We investigate the relationship between~(\ref{cx1}) being an exact complex and~(\ref{cx2}) being exact complexes.
\begin{observation}\label{iscomplex}
If~(\ref{cx1}) is a complex, then the two sequences in~(\ref{cx2}) are also complexes.
\end{observation}
\begin{proof}
We have $d_id_{i+1}=A_i'A'_{i+1}+B'_iB_{i+1}'$.
Assume $d_id_{i+1}=0$. Then $A_i'A_{i+1}'=-B_i'B_{i+1}'$ and therefore the images of both $A_i'A_{i+1}'$ and $B_i'B_{i+1}'$ are contained in $(\delta) R^{b_{i-1}}$. This is equivalent to $A_iA_{i+1}=B_iB_{i+1}=0$.
\end{proof}
 Note that the converse of Observation~(\ref{iscomplex}) is not true, since the images of $A_i'A_{i+1}'$ and $B_i'B_{i+1}'$ can be contained in $(\delta)$, but $A_i'A_{i+1}' \ne -B_i'B_{i+1}'$ (for instance replacing $B_i'$ by $-B_i'$ will cause this to occur).

There will be an additional assumption that we will impose in the course of this investigation, namely
\begin{equation}\label{condition}
(f)  R_0^{b_i}\subseteq \mathrm{im}(\tilde{A}_{i+1})\ \ \ \ \mathrm{and} \ \ \ \  (g) S_0^{b_i}\subseteq \mathrm{im}(\tilde{B}_{i+1})\ \ \ \ \mathrm{for\ all}\  i. \end{equation}

Before stating the results, we illustrate our conclusions using the following two examples:
\begin{example}\label{exnew2}
Consider the rings from Example~(\ref{exnew}). Note that $z_1$ is an exact zero divisor for $R_1$, $z_2$ is an exact zero divisor for $R_2$, and $z_1+z_2$ is an exact zero divisor for $R$.
Consider the following complexes in the roles of the complexes in ~(\ref{cx2})
$$\cdots \rightarrow R_1^{b_{i+1}}\buildrel{z_1}\over\rightarrow R_1^{b_i}\buildrel{z_1}\over\rightarrow R_1^{b_{i-1}}\buildrel{z_1}\over\rightarrow \ldots   \ \ \ \ \ \mathrm{and} \ \ \ \ \ \ \ \cdots \rightarrow S_1^{b_{i+1}}\buildrel{z_2}\over\rightarrow S_1^{b_i}\buildrel{-z_2}\over\rightarrow S_1^{b_{i-1}}\buildrel{z_2}\over\rightarrow \ldots 
$$
and we obtain
$$
\cdots \rightarrow R^{b_{i+1}}\buildrel{z_1+z_2}\over\rightarrow R^{b_i}\buildrel{z_1-z_2}\over\rightarrow R^{b_{i-1}}\buildrel{z_1+z_2}\over\rightarrow  \cdots
$$
in the role of~(\ref{cx1}) (which is a complex). Note that all these complexes are exact, and condition~(\ref{condition}) holds, where $\tilde{A_i}$ is given by multiplication by $z_1$ and $\tilde{B_i}$ is given by multiplication by $z_2$.
\end{example}

\begin{example}\label{ex1}
Consider
$$
R_1=\frac{k[x_1, x_2, y_1, y_2, y_3]}{(x_1, x_2)^2+(y_1, y_2, y_3)^2+x_1(y_1, y_2)}
$$
$$
 S_1=\frac{k[x_3, x_4, x_5, y_4, y_5]}{(x_3, x_4, x_5)^2+(y_4, y_5)^2+y_4(x_3, x_4)}
$$
Construct $R$ as in~(\ref{constr2}), using $f:=x_1y_1, g:=x_4y_4$.

Note that $R_1$ and $S_1$ have exact zero divisors. The following elements are a pair of exact zero divisors in $R_1$:
$$
l_1=x_1+x_2+y_1+y_2+y_3 \ \ \ \ \ l_1'= x_1+x_2-y_1-y_2-y_3
$$ 
and the following elements are a pair of exact zero divisors in $S_1$:
$$
l_2=x_3+x_4+x_5+y_4+y_5, \ \ \ \ \ l_2'=x_3+x_4+x_5-y_4-y_5
$$
(this has been checked using Macaulay 2).
Thus, the complexes
\begin{equation}\label{cc1}
\cdots R_1\buildrel{l_1'}\over\rightarrow R_1  \buildrel{l_1}\over\rightarrow R_1 \buildrel{l_1'}\over\rightarrow R_1 \buildrel{l_1}\over\rightarrow R_1\cdots 
\end{equation}
and 
\begin{equation}\label{cc2}
\cdots S_1 \buildrel{l_2'}\over\rightarrow S_1  \buildrel{l_2}\over\rightarrow S_1 \buildrel{l_2'}\over\rightarrow S_1 \buildrel{l_2}\over\rightarrow S_1\cdots 
\end{equation}
are exact. 

Note that $\tilde{l}_1\tilde{l'}_1=0$, and $\tilde{l}_2\tilde{l'}_2=0$,  so condition~(\ref{condition}) does not hold. In fact, more is true: for every choice of $l_1, l'_1 \in R_1$ and $l_2,  l'_2 \in R_2$ which are pairs of exact zero divisors, we will have $\tilde{l}_1\tilde{l'}_1=\tilde{l}_2\tilde{l}'_2=0$. To see this, write $l_1:=l_{1x}+l_{1y}$, where $l_{1x}$ is a linear combination of $x_1, x_2$, and $l_{1y}$ is a linear combination of $y_1, y_2, y_3$, and note that setting $l_1':=l_{1x}-l_{1y}$ gives $\tilde{l}_1\tilde{l}'_1=0$. Since the annihinlator of $l_1$ is a principal ideal, it follows that $l_1'$ is the generator of that annihilator. A similar argument applies to $l_2$. 

The complexes~(\ref{cc1}) and~(\ref{cc2}) can be used to build a complex of $R$-modules:
\begin{equation}\label{cc3}
\cdots \buildrel{l_1+l_2}\over\rightarrow R\buildrel{l_1'+l_2'}\over\rightarrow R \buildrel{l_1+l_2}\over\rightarrow \cdots 
\end{equation}
 However, this complex is not exact. In fact, $R$ does not have exact zero divisors. To see this, assume that  $L:=L_{\a}+L_{\b} \in R$ is an exact zero divisor, where $L_{\a}$ is a linear combination of $x_1, x_2, y_1, y_2, y_3$, and $L_{\b}$ is a linear combination of $x_3, x_4, x_5, y_4, y_5$. Note $L_{\a }$ and $L_{\b }$ must be nonzero ($x_1, x_2, y_1, y_2, y_3 \in \mathrm{ann}(L_{\b })$, and thus $L_{\b}$ cannot be an exact zero divisor).  Further, write $L_{\a }:=L_{\a x}+L_{\a y}$, where $L_{\a x}$ is a linear combination of $x_1, x_2$, and $L_{\a y}$ is a linear combination of $y_1, y_2, y_3$. Similarly, $L_{\b }:=L_{\b x } + L_{\b y}$. Note that $(x_1, \ldots, x_5)^2R=(y_1, \ldots, y_5)^2 R=0$, and therefore $(L_{\a x} + L_{a y})(L_{\a x} - L_{\a y})=0, (L_{\b x} +L_{\b y})(L_{\b x} - L_{\b y})=0$. Since we also have $$(x_1, x_2, y_1, y_2, y_3)(x_3, x_4, x_5, y_4, y_5)R=0, $$ it follows that $L_{\a x} - L_{\a y} , L_{\b x} - L_{\b y} \in \mathrm{ann}(L)$, thus $\mathrm{ann}(L)$ cannot be a principal ideal. 

We shall see in Lemma~(\ref{lemma_condition}) that the failure of~(\ref{cc3}) to be exact is due to the failure of condition~(\ref{condition}).
We shall see in Corollary~(\ref{thus}) that even though $R$ does not have exact zero divisors, it does have totally reflexive modules of higher rank.

\end{example}

Now we prove that condition~(\ref{condition})  is necessary for~(\ref{cx1}) to be totally acyclic.

\begin{lemma}\label{lemma_condition}
Assume that $R$ is not Gorenstein and~(\ref{cx1}) is a totally acyclic complex. Then (\ref{condition}) holds.
\end{lemma}
\begin{proof}
We know from Theorem~(\ref{Yoshino}) that the betti numbers in a totally acyclic complex are constant, say $b_i=b$, and the entries in the matrices representing the maps $d_i$ are linear.

Let ${\bf u} \in R^b$ be a nonzero vector with linear entries belonging to $\mathfrak{a}$. We claim that $A_i' {\bf u} \ne 0$ for all $i$. In other words, the restriction of  $A_i'$ to the degree one component of $\a ^b$ is injective. Otherwise, we would have ${\bf u} \in \mathrm{ker}(d_i)=\mathrm{im}(d_{i+1})$. Say ${\bf u} = d_{i+1}({\bf e})$, where ${\bf e}\in R^b$ has degree 0. Since $\mathfrak{a}\mathfrak{b}=0$, we have 
$y_1{\bf e}, \ldots y_m{\bf e} \in \mathrm{ker}(d_{i+1})=\mathrm{im}(d_{i+2})$. Say $y_j{\bf e}=d_{i+2}({\bf f_j})$, where ${\bf f_j}\in R^b$ has degree zero. Then $x_l {\bf f_j } \in \mathrm{ker}(d_{i+2})$ for all $1 \le l \le n, 1 \le j \le m$. This shows that $\mathrm{ker}(d_{i+1})$ has at least $m$ minimal generators, and $\mathrm{ker}(d_{i+2})$ has at least $nm$ minimal generators. Continuing along the same lines, we see that $\mathrm{ker}(d_{i+3})$ will have at least $nm^2$ minimal generators, etc. This contradicts the fact that the betti numbers are constant.

Similarly, if ${\bf v} \in R^b$ is a nonzero vector with linear entries belonging to $\mathfrak{b}$, we have $B_i'{\bf v} \ne 0$.

A nonzero vector with linear entries in $R^b$ can be written as ${\bf u}' + {\bf v}'$, where ${\bf u}'$ has entries in $\a$ and ${\bf v}'$ has entries in $\b$. We have 
${\bf u}' +{\bf v}'\in \mathrm{ker}(d_i)$ if and only if $A_i'{\bf u}' =-B_i'{\bf v}'$, and this  is a nonzero vector in $(\delta ) R^b$.
Due to the injectivity of $A_i'$ and $B_i'$, for every $\delta {\bf e} \in (\delta )R^b \cap \mathrm{im}(A_i')\cap \mathrm{im}(B_i')$, there exist unique ${\bf u}', {\bf v}'$ with $A_i'{\bf u}'=-B_i'{\bf v}' =\delta {\bf e} $, and therefore a unique ${\bf u}'+{\bf v}' \in \mathrm{ker}(d_i)$. Since the $\mathrm{ker}(d_i)$ is generated by $b$ linearly independent vectors with linear entries, it follows that $(\delta) R^b \subseteq \mathrm{im}(A_i') \cap \mathrm{im}(B_i')$, which  is equivalent to the desired conclusion.

\end{proof}

\begin{theorem}\label{mainresult}
 Assume that~(\ref{cx1}) is a complex and condition ~(\ref{condition}) holds.
Then (\ref{cx1}) is exact if and only if both of the complexes in (\ref{cx2}) are exact.
\end{theorem}
\begin{proof}
We know from Observation~(\ref{iscomplex}) that the sequences in~(\ref{cx2}) are complexes.

 The hypothesis (\ref{condition}) is equivalent to $(\delta)R^{b_{i-1}}\subseteq \mathrm{im}(A'_i) \cap \mathrm{im}(B'_i)$ for all $i$.
Assume (\ref{cx1}) is exact. We show that the first complex in (\ref{cx2}) is exact. Consider $x \in \mathrm{ker}(A_i)$. We lift $x$ to an element $\tilde{x} \in R_0^{b_i}$, such that $\tilde{A_i}\tilde{x} \in (f) R_0^{b_{i-1}}$. This corresponds to an element $x' \in R^{b_i}$ such that all components of $x'$ are in $\a$, and $A_i'x' \in (\delta )R^{b_{i-1}}\subseteq \mathrm{im}(B'_i)$ by assumption~(\ref{condition}).  Therefore, there exists $y' \in R^{b_i}$ with all components in $\b $ such that $A'_ix'=B'_iy'$. This implies $x'-y' \in \mathrm{ker}(d_i)$. Since (\ref{cx1}) is exact, there exist $x_2'\in R^{b_{i+1}}$ with entries in $\a$ and $y_2' \in R^{b_{i+1}}$ with entries in $\b$ such that 
$$
x' - y' = d_{i+1}(x_2' + y_2')= A_{i+1}'x_2' + B'_{i+1} y_2'
$$
We have $x'-A_{i+1}'x_2' = B'_{i+1}y_2'-y' \in (\delta ) R^{b_i}$. Translating to elements of $R_0^{b_i}$, we have  $\tilde{x}-\tilde{A}_{i+1}\tilde{x}_2 \in (f)R_0^{b_i}$, and therefore $x=A_{i+1}(x_2)$, which is the desired conclusion.

Now assume that the complexes of (\ref{cx2}) are exact. Consider an element $x'+y' \in \mathrm{ker}(d_i)$, where $x'$ has all components in $\a$ and $y'$ has all components in $\b$. We have $d_i(x'+y')=A_i' x'+ B_i'y'$, and therefore $A_i' x'=-B_i' y' \in (\delta) R^{b_0}$.
Translating to elements of $R_0, S_0$, we have $\tilde{A}_i( \tilde{x}) \in (f) R_0^{b_{i-1}}, \tilde{B_i}(\tilde{y}) \in (g)S_0^{b_{i-1}}$, i.e. $x \in \mathrm{ker}(A_i)$ and $y \in \mathrm{ker}(B_i)$. The assumption that the complexes of (\ref{cx2}) are exact implies that there are elements $x_2 \in R_1^{b_{i+1}}, y_2 \in S_1^{b_{i+1}}$ such that 
$x=A_{i+1}(x_2)$ and $y=B_{i+1}(y_2)$. We can lift to elements $\tilde{x}_2\in R_0^{b_{i+1}}, \tilde{y}_2\in S_0^{b_{i+1}}$ such that 
$$
\tilde{x}=\tilde{A}_{i+1}(\tilde{x_2}) \ \ \mathrm{mod} ((f ) R_0^{b_i}), \ \ \ \tilde{y}=\tilde{B}_{i+1}(\tilde{y_2}) \ \ \mathrm{mod} ((g) S_0 ^{b_i})
$$
The assumption (\ref{condition}) allows us to conclude that $\tilde{x} \in \mathrm{im}(\tilde{A}_{i+1}), \tilde{y} \in \mathrm{im}(\tilde{B}_{i+1})$, which translates into $x' \in \mathrm{im}(A'_{i+1}), y' \in \mathrm{im}(B'_{i+1})$, and therefore $x'+y' \in \mathrm{im}(d_{i+1})$.
\end{proof}

The next result allows us to restate condition~(\ref{condition}):
\begin{proposition}\label{this_cond}
Let $R_1=P_1/I_1+(f)$ be a non-Gorenstein quotient of a polynomial ring $P_1$. Assume that $\m _{R_1}^3=0$ and $f$ has degree 2. Assume that there is a minimal totally acyclic complex 
$$
\cdots \rightarrow R_1^b \buildrel{A_{i+1}}\over\rightarrow R_1^b \buildrel{A_i}\over\rightarrow R_1^b \buildrel{A_{i-1}}\over\rightarrow \cdots 
$$
 and let $\tilde{A}_i: R_0^b \rightarrow R_0^b$ be liftings of the maps $A_i$ to $R_0:=P_1/(I_1+\m _{P_1} f)$.

We have
$$
(f) R_0^b \subseteq \mathrm{im}(\tilde{A}_{i-1}) \Leftrightarrow \mathrm{im}(\tilde{A}_{i-1}\tilde{A}_i )=(f)R_0^b
$$
If the above conditions hold, we can construct a minimal totally acyclic complex 
$$
\cdots \rightarrow R_1^b \buildrel{A'_{i+1}}\over\rightarrow R_1^b \buildrel{A'_i}\over\rightarrow R_1^b \buildrel{A'_{i-1}}\over\rightarrow \cdots 
$$
over $R_1$ such that $\tilde{A}'_{i-1}\tilde{A}'_i =fI_b$, where $I_b$ is the identity map on $R_0^b$.
\end{proposition}
\begin{proof}
$(\Leftarrow )$ is obvious. We prove $(\Rightarrow)$. Recall that the matrices $A_i$ have linear entries and every homogeneous element of degree two of $R_1^b$ is in  $\mathrm{ker}(A_{i-2})=\mathrm{im}(A_{i-1})$.  The assumption that $(f) R_0^b \subseteq \mathrm{im}(\tilde{A}_{i-1})$ implies that 
every homogeneous element of degree two of $R_0^b$ is in $\mathrm{im}(\tilde{A}_{i-1})$.

From Theorem~(\ref{Yoshino}), we have $\mathrm{dim}_k([R_1]_2)=\mathrm{dim}_k([R_1]_1)-1$, and therefore $\mathrm{dim}_k([R_0]_2)=\mathrm{dim}_k([R_0]_1)$.

Consider the map of  $k$-vector spaces $L_1: ([R_0]_1)^b \rightarrow ([ R_0]_2)^b$ induced by $\tilde{A}_{i-1}$. We know that this map is surjective, and therefore also  injective. We also have a $k$-linear map $L_0:([R_0]_0)^b \rightarrow ([R_0]_1)^b$ which sends the standard basis vectors to the columns of $\tilde{A}_i$. $L_0$ is also injective, and therefore the composition $L_1L_0: ([R_0]_0)^b \rightarrow ([R_0]_2)^b$ is injective. 
Note that $\mathrm{im}(\tilde{A}_{i-1}\tilde{A}_i ) = \mathrm{im}(L_1L_0)$, and it is contained in $(f) R_0^b$ (since $A_{i-1}A_i =0$). 
Viewing $L_1L_0$ as a map $:[R_0]_0^b \rightarrow (f)R_0^b$, we see that this map is surjective,  because the domain and codomain have the same dimension as vector spaces over $k$.

To prove the last statement, note that we have $\tilde{A}_i\tilde{A}_{i+1}=fU_i$ where $U_i : R_0^b \rightarrow R_0^b$ are invertible. We define
$\tilde{A}'_i:=V_i \tilde{A}_i W_i$ where $V_i, W_i: R_0^b \rightarrow R_0^b$ are invertible. For $i=0$, we let $V_0, W_0=I_b$. For $i>0$, we define $V_i, W_i$ recursively as follows: $V_{i+1}:=W_i^{-1}, W_{i+1}:=(V_iU_i)^{-1}$. For $i<0$, say $i=-j$, we define $V_{-j}, W_{-j}$ recursively as follows: $V_{-j-1}:=(U_{-j}W_{-j})^{-1}, W_{-j-1}:=V_{-j}^{-1}$. We now have $\tilde{A}'_i \tilde{A}'_{i+1}=fI_b$ for all $i$. The complex with the maps $A'_i$ (where $A'_i: R_1 ^b \rightarrow R_1^b$ is obtained from $\tilde{A}_i$ by modding out $f$) is still totally acyclic because the operations involved in constructing $A'_i$ from $A_i$ do not change the dimensions of the kernel and the image.

\end{proof}


\begin{corollary}\label{conclusion}
Let $R_1, S_1$ be non-Gorenstein rings with $\m _{R_1}^3=\m_{S_1}^3=0$, and let $f, g$ be part of minimal systems of generators for the defining ideals of $R_1$, respectively $S_1$. Let $R$ be constructed as in~(\ref{constr2}). Assume that $R$ is not Gorenstein.

Then $R$ has minimal totally acyclic complexes  if and only if both $R_1$ and $S_1$  have minimal totally acylclic complexes such that conditions (\ref{condition}) are satisfied. 
\end{corollary}
\begin{proof}
Assume that $R$ has a minimal totally acyclic complexes. Then the conclusion follows immediately from Theorem~(\ref{mainresult}) and Lemma~(\ref{condition}).

Conversely, assume that $R_1$ and $S_1$ admit minimal totally acyclic complexes such that condition~(\ref{condition}) is satisfied. Replacing each of these complexes by direct sums of copies of themselves if necessary,
we may assume that the free modules in both complexes have the same rank (condition~(\ref{condition}) will continue to hold). Let $A_i :R_1^b \rightarrow R_1^b$ denote the maps in a minimal totally acyclic complex over $R_1$, and let $B_i: S_1^b\rightarrow S_1^b$ be the maps in the complex over $S_1$. 
It follows from  Proposition~(\ref{this_cond}) that we may assume $\tilde{A}_{i-1} \tilde{A}_i = f I_{R_0^b}$ and $\tilde{B}_{i-1}\tilde{B_i} =-gI_{S_0^b}$, where $I_{R_0^b}$, $I_{S_0^b}$ denote the identity functions on these modules.

We have established  in Setup~(\ref{st-up}) that the maps $A_i$ and $B_i$ can be used to construct $d_i: R^b \rightarrow R^b$, $d_i=A_i'+B_i'$. Since $d_{i-1}d_i=A_{i-1}'A'_i-B'_{i-1}B'_i=0$, these maps give rise to a complex of free $R$-modules. Theorem~(\ref{mainresult}) now tells us that this complex is totally acyclic.
\end{proof}

As we have seen in Example~(\ref{ex1}),  the hypothesis~(\ref{condition}) cannot be omitted in the statement of Theorem~(\ref{mainresult}).

 The next example shows that it is possible for $R_1, S_1$ to have minimal totally acyclic complexes, but for the ring $R$ constructed as in ~(\ref{constr2}) to not have any.
\begin{example}
Let
$$
R_1=\frac{k[x_1,y_1,z_1]}{(x_1^2, y_1^2, z_1^2-x_1y_1, x_1z_1, y_1z_1)}, S_1=\frac{k[x_2,y_2,z_2]}{(x_2^2, y_2^2, z_2^2-x_2y_2, x_2z_2, y_2z_2)}
$$
Construct $R$ as in ~(\ref{constr2}), using any choice of $f_1, g_1$ from a minimal system of generators for the defining ideals of $R_1$ and $S_1$.

Note that $R_1, S_1$ are Gorenstein, and therefore they have minimal totally acyclic complexes. However, $\mathrm{dim}_k[(R]_1)=6$ and $\mathrm{dim}_k([R]_2)=3\ne \mathrm{dim}_k([R]_1)-1$, so $R$ does not have minimal totally acyclic complexes by Theorem~(\ref{Yoshino}).
\end{example}

We do not know any examples of non-Gorenstein rings $R_1, S_1$ with $\m_{R_1}^3=\m_{S_1}^3=0$ that have minimal totally acyclic complexes such that the ring $R$  constructed as in ~(\ref{constr2}) does not.

\section{Totally acyclic complexes with prescribed liftings}
Let $R_1=P/I+(f)$ denote a quotient of a polynomial ring $P=k[x_1, \ldots, x_n]$ with $\m_{R_1}^3=0$. Assume that $R_1$ is not Gorenstein and has minimal totally acyclic complexes. Let $R_0=P/(I+\m _{P_1}f)$.

The results of the previous section prompt us to ask the following:

\begin{question}
Is there a minimal totally acyclic complex
$$
\cdots R_1^b \buildrel{A_i}\over\longrightarrow R_1^b \buildrel{A_{i-1}} \over\longrightarrow R_1^b \cdots
$$
such that 
\begin{equation}\label{cond3}
(f)R_0^b \subseteq \mathrm{im}(\tilde{A}_{i-1}\tilde{A}_i) \  \ \forall i \ \ \ ?
\end{equation}
Here,  $\tilde{A}_i$ denotes a lifting of $A_i$ to $R_0$.
\end{question}
Example~(\ref{ex1}) shows that it is possible for $R_1$ to have a minimal totally acyclic complex consisting of modules of rank $b=1$,  but not have any such complex (with free modules of the same rank) satisfying~(\ref{cond3}). However, if we are willing to increase the rank of the free modules in the complex (and under additional assumptions on the minimal totally acyclic complex)  we have the following:

\begin{theorem}\label{specify}
Let $R_1=P/I+(f), R_0=P/(I+\m _P f)$ be as above, where $P$ is a polynomial ring over an algebraically closed field $k$. Assume that $R_1$ has a minimal totally acyclic complex which is periodic with period two, i.e. it has the form
$$
\cdots \rightarrow  R_1^b \buildrel{X}\over\rightarrow R_1^b \buildrel{W}\over\rightarrow R_1^b \buildrel{X}\over\rightarrow R_1^b \buildrel{W}\over\rightarrow \cdots 
$$
Moreover, assume that 
\begin{equation}\label{extra}\tilde{X}\tilde{W}=\tilde{W}\tilde{X},
\end{equation}
 where $\tilde{X}, \tilde{W}$ denote liftings of $X, W$ to $R_0$.

 Assume that $f=y_1z_1+\ldots + y_kz_k$, where $y_i, z_i \in R_0$ are linear.
Then there is a totally acyclic complex
\begin{equation}\label{tacomplex}
\ldots \rightarrow R_1^{2^k b} \buildrel{A}\over\rightarrow R_1^{2^k b}\buildrel{B}\over\rightarrow R_1^{2^k b}\buildrel{A}\over\rightarrow \ldots
\end{equation}
such that 
\begin{equation}\label{here}
(f) R_0^{2^k b} \subseteq \mathrm{im}(\tilde{A}\tilde{B})\cap \mathrm{im}(\tilde{B}\tilde{A}), 
\end{equation}
where $\tilde{A}, \tilde{B}$ denote  matrices with entries in $R_0$ obtained by lifting each entry of $A$, respectively $B$, to $R_0$.
\end{theorem}
\begin{proof}
Since $XW=WX=0$, the matrices representing $\tilde{X}\tilde{W}$ and $\tilde{W}\tilde{X}$ have entries in $(f)$.
By choosing bases, we may assume that 
$$\tilde{X}\tilde{W}=\tilde{W}\tilde{X}=\mathrm{diag}(0, \ldots,  0, f,  \ldots, f),$$ with the last $b-v$ diagonal entries equal to $f$. If $v=0$, there is nothing to show. Assume $v>0$.
For each $1\le j \le k$, define $Y'_j$ and $ Z'_j$ to be the $b\times b$ matrices $Y'_j :=\mathrm{diag}(y_j, \ldots, y_j, 0, \ldots, 0)$, with $v$ diagonal entries equal to $y_j$, and $Z_j':=\mathrm{diag}(z_j, \ldots, z_j, 0, \ldots, 0)$ with $v$ diagonal entries equal to $z_j$. Let $Y_j=\mathrm{diag}(Y_j', \ldots, Y_j')$, $Z_j=\mathrm{diag}(Z_j', \ldots, Z_j')$ consisting of $j$ diagonal bloks equal to $Y_j'$ and respectively $Z_j'$.

Let $\alpha \in k$. For $1 \le j \le k$, we define $2^j b \times 2^j b$ matrices $\tilde{A}_j, \tilde{B}_j$ recursively as follows:
$$
\tilde{A}_1=\left(\begin{array}{cc} \tilde{X} & \alpha Y_1 \\ -\alpha Z_1 & \tilde{W} \\ \end{array}\right), \ \ \ \ \ \  \tilde{B}_1=\left(\begin{array}{cc} \tilde{W} & -\alpha Y_1\\  \alpha Z_1& \tilde{X}\\ \end{array}\right)
$$
$$\tilde{A}_{j+1}=\left(\begin{array}{cc} \tilde{A}_j & \alpha Y_{j+1}  \\ -\alpha Z_{j+1} & \tilde{B}_j\\ \end{array} \right), \ \ \ \ \  \tilde{B}_{j+1}=\left(\begin{array}{cc} \tilde{B}_j & -\alpha Y_{j+1} \\ \alpha Z_{j+1} & \tilde{A}_j \\ \end{array}\right)
$$
We see that
$$
\tilde{A}_1\tilde{B}_1=\tilde{B}_1\tilde{A}_1=\left(\begin{matrix} \tilde{X}\tilde{W} +\alpha^2 Y_1Z_1 & 0 \\  0 & \tilde{W}\tilde{X}+\alpha^2 Y_1Z_1\end{matrix}\right),
$$
and we get by induction that
$$\tilde{A}_{j}\tilde{B}_{j}=\tilde{B}_j\tilde{A}_j=\mathrm{diag}(\Delta_j, \ldots, \Delta _j),$$
where $\Delta _j$ is the $b \times b$ matrix
$\Delta _j=  \mathrm{diag}(\alpha^2 \sum_{i=1}^j y_iz_i , \ldots, \alpha ^2 \sum_{i=1}^j y_i z_i, f, \ldots, f)$ (with the last $b-v$ entries of each block being equal to $f$), and there are $2^j$ blocks equal to $\Delta _j$ along the diagonal.

In particular, $\tilde{A}_k \tilde{B}_k=\tilde{B}_k\tilde{A}_k$ consists of $2^k$ blocks of size $b \times b$  equal to $\mathrm{diag}(\alpha ^2f , \ldots, \alpha ^2 f, f, \ldots, f)$ along the diagonal, and zeroes otherwise. 

Letting $A$ and $B$ be the matrices obtained by taking the images of the entries of $\tilde{A}_k$ and $\tilde{B}_k$ respectively in $R_0$, it is now clear that~(\ref{tacomplex}) is a complex over $R_0$, and condition~(\ref{here}) is satisfied if $\alpha \ne 0$.

It remains to prove that there are choices of $\alpha \ne 0$ such that~(\ref{tacomplex}) is totally acyclic.

It was shown in  \cite{AV} , Theorem 5.1 that there is a countable intersection $\mathcal{U}$  of nonempty Zariski open sets  in $k={\bf A}_k^1$ such that~(\ref{tacomplex}) is totally acyclic if and only if $\alpha \in \mathcal{U}$. Due to the periodic nature of the complex~(\ref{tacomplex}), in this case we may take $\mathcal{U}$ to be a finite intersection of Zariski open sets. We summarize the argument from \cite{AV} for the convenience of the reader.

Note that $A$ and $B$ give rise to  $k$-linear maps $A',  B': [R_1]_1 ^D\rightarrow [R_1]_2^D$, where $D=2^kb$. The condition that these $k$-linear maps have maximal rank  can be described as the  non-vanishing of certain minors (after choosing vector space bases for $[R_1]_1^D$ and $[R_1]_2^D$), and therefore are open conditions in terms of $\alpha $. Having maximal rank is equivalent to surjectivity, and, recalling that $\mathrm{dim}_k([R_1]_2)=\mathrm{dim}_k([R_1]_1)-1$, it is also equivalent to the fact that the kernel of the $k$-linear maps is $D$-dimensional. Since we have 
$\mathrm{im}(B)\subseteq \mathrm{ker}(A)$ and $\mathrm{im}(A)\subseteq \mathrm{ker}(B)$, this is equivalent to exactness of the complex~(\ref{tacomplex}) (note that $[\mathrm{ker}(A)]_1=\mathrm{ker}(A')$, and $[\mathrm{ker}(A)]_2=([R_1]_2)^D$).

Similar open conditions imposed on the transpose matrices $A^t$ and $B^t$ ensure the acyclicity of the dual complex. 

These open sets are non-empty because~(\ref{tacomplex}) is totally acyclic for $\alpha =0$.
\end{proof}
\begin{corollary}\label{thus}
Let $R_1$, $S_1$, $R$ be as in the hypothesis of Corollary~(\ref{conclusion}). Assume that $R_1$ and $S_1$ have minimal totally acyclic complexes that are periodic of period two, and condition~(\ref{extra}) is satified (for instance, this holds if $R_1$ and $S_1$ have exact zero divisors).
Then $R$ has minimal totally acyclic complexes.
\end{corollary}
\begin{note}
Note that the complex~(\ref{tacomplex}) constructed in the proof of Theorem ~(\ref{specify}) under the assumption that $R_1$ has a pair of exact zero divisors is periodic with period two and satisfies condition~(\ref{extra}). Therefore one may start with rings $R_1, S_1$ as above, and construct a family of rings that have minimal totally acyclic complexes by iterating the construction of~(\ref{constr2}).
\end{note}
\begin{example}\label{finalex}
We illustrate the construction of totally acyclic complexes for rings as in Corollary~(\ref{thus}) in the case of the ring $R$ from Example~(\ref{ex1}).
Recall that $R$ was constructed as a connected sum of $R_1, S_1$, where $R_1$ has a pair of exact zero divisors $l_1=x_1+x_2+y_1+y_2+y_3$, $l_1'=x_1+x_2-y_1-y_2-y_3$, and $S_1$ has a pair of exact zero divisors $l_2= x_3+x_4+x_5-y_4-y_5$. We use the construction given in the proof of Corollary~(\ref{conclusion}) to obtain a totally acyclic complex over $R$. The first step is to find totally acyclic complexes over $R_1$ and $S_1$ that satisfy condition~(\ref{condition}). Using the procedure described in the proof of Theorem~(\ref{specify}), we find that
$$
\cdots \rightarrow R_1^2 \buildrel{X_1}\over\rightarrow R_1^2 \buildrel{W_1}\over\rightarrow R_1^2 \buildrel{X_1}\over\rightarrow R_1^2 \buildrel{W_1}\over\rightarrow R_1^2 \rightarrow \cdots
$$
and
$$
\cdots \rightarrow S_1^2\buildrel{X_2}\over\rightarrow S_1^2\buildrel{W_2}\over\rightarrow S_1^2 \buildrel{X_2}\over\rightarrow S_1^2 \buildrel{W_2}\over\rightarrow S_1^2 \rightarrow \cdots
$$
satisfy these requirements, where
$$
X_1=\left( \begin{array}{cc} l_1& x_1 \\ -y_1 & l_1'\\ \end{array}\right), \ \ \ \ W_1=\left(\begin{array}{cc} l_1' & -x_1 \\ y_1 & l_1\\ \end{array}\right)
$$
$$
X_2 = \left( \begin{array}{cc} l_2 & x_4 \\ -y_4 & l_2' \\ \end{array} \right), \ \ \  W_2=\left(\begin{array}{cc} l_2'  & -x_4\\ y_4 & l_2\\ \end{array}\right)
$$
More precisely, we have $\tilde{X}_1\tilde{W}_1=\tilde{W}_1\tilde{X}_1 = f I_{R_1^2}$, and $\tilde{X}_2\tilde{W}_2=\tilde{W}_2 \tilde{X_2}=gI_{S_1^2}$.
This wil ensure that using $d:=X_1+W_1, d':=X_2-W_2$ gives a complex 
$$
\cdots \rightarrow R^2 \buildrel{d}\over\rightarrow R^2 \buildrel{d'}\over\rightarrow R^2 \buildrel{d}\over\rightarrow R^2 \buildrel{d'}\over\rightarrow R^2 \rightarrow \cdots
$$
and Theorem~(\ref{mainresult}) shows that this complex is exact. The same reasoning applies for the dual; therefore this is a totally acyclic complex over $R$.
\end{example}


\begin{thebibliography}{99}

\bibitem{AB} M. Auslander and M. Bridger, {\em Stable module theory}, Memoirs of the American Mathematical Society no. 94, American Mathematical Society, Providence R.I.1969

\bibitem{AAM} H. Ananthnarayan, L. Avramov, and F. Moore, {\em Connected sums of Gorenstein rings}, Crelle's Journal {\bf 667} (2012), 149--176.

\bibitem{ACLY} H. Ananthnarayan, E. Celikbas, J. Laxmi, Z. Yang, {\em Decomposing Gorenstein rings as connected sums},  arXiv:1406.7600

\bibitem{AV} C. Atkins and A. Vraciu, {\em On the existence of non-free totally reflexivem modules}, J. of Commutative Algebra, to appear.

\bibitem{CLW} E. Celikbas, J. Laxmi, J. Weyman, {\em Embeddings of canonical modules and resolutions of connected sums},  arXiv:1704.03072

\bibitem{CPST} L. W. Christensen, G. Piepmeyer, J. Striuli and R. Takahashi, {\em Finite Gorenstein representation type implies simple singularity}, Adv. Math. {\bf 218} (2008), no. 4, 1012--1026.

\bibitem{Kustin-V}, A. Kustin and A. Vraciu {\em Totally reflexive modules over rings that are close to Gorenstein}, J. of Algebra, to appear. 

\bibitem{NSW} S. Nasseh and S. Sather-Wagstaff, {\em Vanishing of Ext and Tor over fiber products}, Proceedings of the Amer. Math. Soc. {\bf 145} (2017) , no, 11, 4661--4674.

\bibitem{Yo} Y. Yoshino, {\em Modules of G-dimension zero over local rings with the cube of the maximal ideal being zero}, Commutative algebra, singularities and computer algebra, (Sinaia 2002), 255--273, NATO Sci. Ser. II Math. Phys. Chem. {\bf 115}, Kluwer Acad. Publ., Dordrecht, 2003.

\end{thebibliography}
\end{document}